\def\draft{n}
\documentclass{amsart}
\usepackage[headings]{fullpage}
\usepackage{amssymb,epic,eepic,epsfig}
\usepackage{pb-diagram,lamsarrow,pb-lams}

\theoremstyle{plain}

\newtheorem{theorem}{Theorem}
\newtheorem{proposition}{Proposition}[section]
\newtheorem{lemma}[proposition]{Lemma}
\newtheorem{corollary}[proposition]{Corollary}

\newtheorem{conjecture}{Conjecture}

\theoremstyle{definition}
\newtheorem{definition}[proposition]{Definition}
\newtheorem{problem}[proposition]{Problem}

\theoremstyle{remark}
\newtheorem{example}[proposition]{Example}

\newtheorem{remark}[proposition]{Remark}

\def\printname#1{
        \if\draft y
                \smash{\makebox[0pt]{\hspace{-0.5in}
                        \raisebox{8pt}{\tt\tiny #1}}}
        \fi
}

\newlength{\standardunitlength}
\catcode`\@=11
\long\def\@makecaption#1#2{%
     \vskip 10pt

\setbox\@tempboxa\hbox{
       \small\sf{\bfcaptionfont #1. }\ignorespaces #2}%
     \ifdim \wd\@tempboxa >\captionwidth {%
         \rightskip=\@captionmargin\leftskip=\@captionmargin
         \unhbox\@tempboxa\par}%
       \else
         \hbox to\hsize{\hfil\box\@tempboxa\hfil}%
     \fi}
\font\bfcaptionfont=cmssbx10 scaled \magstephalf
\newdimen\@captionmargin\@captionmargin=2\parindent
\newdimen\captionwidth\captionwidth=\hsize
\catcode`\@=12

\def\lbl#1{\label{#1}\printname{#1}}

\def\BN{\mathbb N}
\def\BZ{\mathbb Z}

\def\BQ{\mathbb Q}
\def\BR{\mathbb R}
\def\BC{\mathbb C}

\def\calP{\mathcal P}

\def\a{\alpha}

\def\e{\epsilon}

\def\b{\beta}

\def\sub{\subset}

\def\longto{\longrightarrow}

\def\pt{\partial}

\def\calB{\mathcal{B}}

\def\Log{\mathrm{Log}}

\def\Li{\mathrm{Li}}
\def\calB{\mathcal{B}}

\def\Lp{L^{\mathrm{p}}}
\def\Lnp{L^{\mathrm{np}}}

\def\ft{\mathfrak{t}}

\def\Kind{K^{\mathrm{ind}}}
\def\CH{\mathrm{CH}}
\def\FT{\mathrm{FT}}
\def\Log{\mathrm{Log}}
\def\Spec{\mathrm{Spec}}

\def\sub{\subset}
\def\qbinom#1#2{\binom{#1}{#2}_q}
\def\hatCP{\widehat{X_{\ft}}}
\def\CV{\mathrm{CV}}
\def\hatBC{\widehat{\calB(\BC)}}
\def\qterm{$q$-term}
\def\sqterm{special $q$-term}
\def\Sqterm{Special $q$-term}
\def\calD{\mathcal{D}}

\def\BCs{\mathbb{\BC}^{**}}
\def\hatFT{\widehat{\mathrm{FT}}}

\def\Ker{\mathrm{Ker}}
\def\QL{\mathrm{QL}}

\begin{document}


\title[$q$-terms, singularities and the extended Bloch group]{
$q$-terms, singularities and the extended Bloch group}
\author{Stavros Garoufalidis}
\address{School of Mathematics \\
         Georgia Institute of Technology \\
         Atlanta, GA 30332-0160, USA \\ 
         {\tt http://www.math.gatech} \newline {\tt .edu/$\sim$stavros } }
\email{stavros@math.gatech.edu}

\thanks{The author was supported in part by NSF. \\
\newline
1991 {\em Mathematics Classification.} Primary 57N10. Secondary 57M25.
}

\date{August 31, 2010}

\dedicatory{Dedicated to S. Bloch on the occasion of his sixtieth birthday.
}

\begin{abstract}
Our paper originated from a generalization of the Volume Conjecture to
multisums of $q$-hypergeometric terms. This generalization was 
sketched by Kontsevich in a problem list in Aarhus University 
in 2006; \cite{Ko}. We introduce the notion of a 
$q$-hypergeometric term (in short, $q$-term). The latter 
is a product of ratios of $q$-factorials in linear forms in several variables.

In the first part of the paper, we show how to construct elements of the 
Bloch group (and its extended version) given a \qterm. Their image
under the Bloch-Wigner map or the Rogers dilogarithm is a finite set
of periods of weight $2$, in the sense of Kontsevich-Zagier.

In the second part of the paper we introduce the notion of a special 
$q$-term, its corresponding sequence of polynomials, and its generating
series. Examples of special $q$-terms come naturally from Quantum Topology,
and in particular from planar projections of knots. 

The two parts are tied together by a conjecture that relates the singularities
of the generating series of a special $q$-term with the periods of the
corresponding elements of the extended Bloch group. In some cases (such as
the $4_1$ knot), the conjecture is known.
\end{abstract}

\maketitle

\tableofcontents

\section{Introduction}
\lbl{sec.intro}

\subsection{A brief summary of our results}
\lbl{sub.brief}

Our paper originated from a generalization of the Volume Conjecture to
multisums of $q$-hypergeometric terms. This generalization was 
sketched by Kontsevich in a problem list in Aarhus University 
in 2006; \cite{Ko}. Our paper expands Kontsevich's
problem, and reveals a close and precise relation between 
special $q$-hypergeometric terms (defined below), elements of the extended 
Bloch group and its conjectural relation to singularities of generating
series. Oddly enough, setting $q=1$ also provides a relation between
special hypergeometric terms and elements of the {\em extended additive} 
Bloch group. This is explained in a separate publication; see \cite{Ga1}
and \cite{Ga2}.

Our paper consists of two, rather disjoint parts, that are tied together 
by a conjecture. 

In the first part, we introduce the notion of a \qterm , and assign to it
elements of the Bloch group and its extended version. The image of
these elements of the extended Bloch group under the Rogers dilogarithm
is a finite set of periods in the sense of Kontsevich-Zagier.

In the second part of the paper we introduce the notion of \sqterm , 
its corresponding sequence of polynomials, and its generating
series. Examples of \sqterm s come naturally from Quantum Topology,
and in particular from planar projections of knots. 

Finally, we formulate a conjecture that relates the singularities
of the generating series of a special $q$-term with the periods of the
corresponding elements of the extended Bloch group. In some cases (such as
the $4_1$ knot, the conjecture is known, and implies the Volume Conjecture
to all orders, with exponentially small terms included.

\subsection{What is a \qterm?}
\lbl{sub.goal}

The next definition, taken from \cite{WZ}, plays a key role in our paper.

\begin{definition}
\lbl{def.qterm}
A {\em \qterm } $\ft$ in the $r+1$ variables 
$k=(k_0,k_1,\dots,k_r) \in \BN^{r+1}$ is a list that consists of
\begin{itemize}
\item
An integral symmetric quadratic form $Q(k)$ in $k$,
\item
An integral linear form $L$ in $k$ and a vector $\e=(\e_0,\dots,\e_r)$
with $\e_i=\pm 1$ for all $i$,
\item
An integral linear form $A_j$ in $k$ for $j=1,\dots,J$
\end{itemize}
\end{definition}
A \qterm\ $\ft$ gives rise to an expression of the form
\begin{equation}
\lbl{eq.qterm}
\ft_{k}(q)=q^{Q(k)} \e^{L(k)} \prod_{j=1}^{J}  (q)^{\e_j}_{A_j(k)} \in \BQ(q)
\end{equation}
valid for $k \in \BN^{r+1}$ such that $A_j(k) \geq 0$ for all $j=1,\dots,J$.
Here, the standard $q$-{\em factorial} of a natural number $n$
and the $q$-{\em binomial} for $0 \leq m \leq n$
defined by (see for example \cite{St}):

\begin{equation}
\lbl{eq.qbinomial}
(q)_n=\prod_{j=1}^n (1-q^j), \qquad 
\qbinom{n}{m}=\frac{(q)_n}{(q)_m (q)_{n-m}}
\end{equation}
\qterm s can be easily constructed and may be combinatorially 
encoded by matrices of the coefficients of the forms $A_j,L,Q$
much in the spirit of Neumann-Zagier and Nahm; see \cite{NZ, Ko,Na}.

\subsection{The Bloch group}
\lbl{sub.bloch}

Let us recall the symbolic-definition of the Bloch group from 
\cite{B1}; see also \cite{DS,Ne2,Su1,Su2}.
Below $F$ denotes a field of characteristic zero. 

\begin{definition}
\lbl{def.bloch}
\rm{(a)}
The {\em pre-Bloch group} 
$\calP(F)$ is the quotient of the freeabelian group 
generated by symbols $[z]$, $z \in F\setminus\{0,1\}$, 
subject to the {\em 5-term relation}:
\begin{equation}
\lbl{eq.5term}
[x]-[y]+\left[\frac{y}{x}\right]-\left[\frac{1-x^{-1}}{1-y^{-1}}\right]+
\left[\frac{1-x}{1-y}\right] =0.
\end{equation}
\rm{(b)} 
The {\em Bloch group} $\calB(F)$ is the kernel of the 
homomorphism
\begin{equation}
\lbl{eq.nu}
\nu: \calP(F) \longto F^* \wedge F^*, \qquad
[z] \mapsto z \wedge (1-z)
\end{equation}
to the second exterior power of the abelian group $F^*$ defined by
mapping a generator $[z]$ to $z \wedge(1-z)$. The second exterior power
$ G \wedge G$ of an abelian group $G$ is defined by:
\begin{equation}
\lbl{eq.lambda2}
G \wedge G= G \otimes_{\BZ} G/(a \otimes b + b \otimes a)
\end{equation}
\end{definition}

Recall the {\em Bloch-Wigner function} (see for example, \cite[Eqn.39]{NZ})

\begin{equation}
\lbl{eq.bw}
\calD_2: \BC \longto i \BR, \qquad
z \mapsto \calD_2(z):=i \Im(\Li_2(z)+\log(1-z) \log|z|),
\end{equation}
which is continuous on $\BC$ and analytic on $\BC\setminus\{0,1\}$.
Since $\calD_2$ satisfies the 5-term relation, we can define a 
map on the pre-Bloch group of the complex numbers:

\begin{equation}
\lbl{eq.rp}
R: \calP(\BC) \longto i \BR, \qquad z \mapsto \calD_2([z])
\end{equation}
and its restriction to the Bloch group (denoted by the same notation):

\begin{equation}
\lbl{eq.r}
R: \calB(\BC) \longto i \BR, \qquad z \mapsto \calD_2([z])
\end{equation}

\subsection{From \qterm s to the Bloch group}
\lbl{sub.special2bloch}

In order to state our results, let us introduce some useful notation. 
Given a linear form $A$ in $r+1$ variables $k=(k_0,k_1,\dots,k_r)$ 
let us define

\begin{equation}
\lbl{eq.vi}
v_i(A)=a_i,  \qquad 
\text{where} \qquad A(k)=\sum_{i=0}^r a_i k_i.
\end{equation}
If $z=(z_0,\dots,z_r) \in \BC^*$ and $A$ is an integral linear form $A$ 
in $k=(k_0,k_1,\dots,k_r)$, let us abbreviate

\begin{equation}
\lbl{eq.zA}
z^A=\prod_{i=0}^r z_i^{v_i(A)}.
\end{equation}
If $Q$ is a symmetric quadratic form in $k=(k_0,k_1,\dots,k_r)$, we can
write it in the form:
$$
Q(k)=\frac{1}{2} \sum_{i,j=0}^r Q_{ij} k_i k_j + \sum_{i=0}^r \QL_i k_i.
$$
$Q$ gives rise to the linear forms $Q_i$ (for $i=0,\dots,r$) in $k$ 
defined by
$$
Q_i(k)=\sum_{j=0}^r Q_{ij} k_j.
$$

The next definition associates to a \qterm\ $\ft$ the complex points
of an affine scheme defined over $\BQ$.

\begin{definition}
\lbl{def.var}
Given a \qterm\ $\ft$, let
$X_{\ft}$ denote the set of points $z=(z_0,\dots,z_r) \in (\BCs)^{r+1}$ that
satisfy the following system of {\em Variational
Equations}:
\begin{equation}
\lbl{eq.var}
z^{Q_i} \e^{v_i(L)} 
\prod_{j=1}^{J} (1-z^{A_j})^{v_i(\e_j A_j)}=1
\end{equation}
for $i=0,\dots,r$. Here, $\BCs=\BC\setminus\{0,1\}$.
\end{definition}
Generically, $X_{\ft}$ is a zero-dimensional set. In that case, 
$X_{\ft} \subset (\overline{\BQ}^*)^{r+1}$, where $\overline{\BQ}$ is the
set of algebraic numbers. The next definition assigns elements of 
the pre-Bloch group to a \qterm\ $\ft$.

\begin{definition}
\lbl{def.ebeta}
Given a \qterm\ $\ft$, consider the map:
\begin{equation}
\lbl{eq.ebeta}
\b_{\ft}: X_{\ft} \longto \calP(\BC)
\end{equation}
given by:
\begin{equation}
\lbl{eq.ebetab}
z \mapsto \b_{\ft}(z):=
\sum_{j=1}^{J} \e_j [z^{A_j}].
\end{equation}
\end{definition}

The next theorem, communicated to us by Kontsevich, assigns elements of the 
Bloch group to a \qterm\ $\ft$. 
In \cite{Za2}, Zagier attributes the result to Nahm \cite{Na}.

\begin{theorem}
\lbl{thm.0}
\rm{(a)} The map $\b_{\ft}$ descends to a map 
\begin{equation}
\lbl{eq.ebeta2}
\b_{\ft}: X_{\ft} \longto \calB(\BC)
\end{equation}
which we denote by the same name.
\newline
\rm{(b)} The image of 
$$
R \circ \b_{\ft}: X_{\ft} \longto i \BR
$$
is a finite subset of $i \BR \cap \calP$, where $\calP$ is the set of periods
in the sense of Kontsevich-Zagier; \cite{KZ}.
\end{theorem}

\subsection{An example}
\lbl{sub.example}

We illustrate Theorem \ref{thm.0} with the following example 
which is nontrivial and of interest to 
Quantum Topology. Consider the \sqterm\ 

\begin{equation}
\lbl{eq.abe1}
\ft_n(q)=q^{a \frac{n(n+1)}{2}} \e^n (q)_n^b 
\end{equation}
where $a \in \BZ$, $\e=\pm 1$ and $b \in \BN$. In other words, 
$$
r=0, \qquad k=(n), \qquad
Q(n)= a \frac{n(n+1)}{2}, \qquad J=b, \qquad
A_j(n)=n, j=1,\dots,b.
$$
Then,
\begin{equation}
\lbl{eq.abe2}
X_{\ft}=\{ z \in \BC^* | z^a (1-z)^b \e=1 \}
\end{equation}
The map \eqref{eq.ebeta} is given by: 

\begin{equation}
\lbl{eq.abe3}
X_{\ft} \longto \calB(\BC), \qquad z \mapsto \b_{\ft}(z):=a [z].
\end{equation}
Equation \eqref{eq.abe2} and the corresponding element
of the Bloch group was also studied by Lewin, using his method of 
{\em ladders}. In that sense, the Variational Equations
\eqref{eq.var} is a generalization of the method of ladders.

Theorem \ref{thm.0} is a practical way of constructing 
elements of the Bloch group $\calB(\BC)$ that are typically defined
over number fields. Even if $X_{\ft}$ is not $0$-dimensional, its image
under $R \circ \b_{\ft}$ always is finite.
This finiteness is positive evidence for Bloch's {\em Rigidity Conjecture}, 
which states that
$$
\calB(\overline{\BQ}) \otimes \BQ \cong \calB(\BC) \otimes \BQ.
$$

\subsection{Acknowledgement}
An early version of this paper was discussed with M. Kontsevich in the
IHES and University of Miami and with S. Bloch in Chicago. 
The author wishes to thank S. Bloch and M. Kontsevich for generously
sharing their ideas with the author and for their hospitality in the IHES 
in the fall of 2006 and in Chicago in the spring of 2007. 
The paper was finished in the spring of 2007, and updated further in the
summer of 2007 to incorporate the final extension of the Rogers dilogarithm
due to Goette and Zickert. The author wishes to thank D. Thurston and 
C. Zickert for many stimulating conversations.

\section{Proof of Theorem \ref{thm.0}}
\lbl{sec.pfthm0}

This section is devoted to a proof of Theorem \ref{thm.0}, using 
elementary symbol-like manipulations. See also \cite[p.43]{Za2}. Zagier
attributes the following proof to Nahm.

\begin{proof}(of Theorem \ref{thm.0})
Let us fix $z=(z_0,\dots,z_r)$ that satisfies the Variational Equations
\eqref{eq.var}. 
We will show that $\nu(\b_{\ft}(z))=0 \in \BC^* \wedge \BC^*$, where $\nu$
is given by \eqref{eq.nu}.  We compute as follows:

\begin{eqnarray*}
\b_{\ft}(z) & = & 
\sum_{j=1}^{J}  
\e_j ( z^{A_j} \wedge (1-z^{A_j}) ).
\end{eqnarray*}
On the other hand, we have:
$$
z^{A_j}=\prod_{i=0}^r z_i^{v_i(A_j)}.
$$
Thus,
\begin{eqnarray*}
\e_j (z^{A_j} \wedge (1-z^{A_j})) &=& \e_j \sum_{i=0}^r z_i^{v_i(A_j)} \wedge 
(1-z^{A_j}) 
\\
&=& \sum_{i=0}^r \e_j v_i(A_j) ( z_i \wedge (1-z^{A_j})) \\
&=& \sum_{i=0}^r z_i \wedge \left((1-z^{A_j})^{v_i(\e_j A_j)}\right).
\end{eqnarray*}
Since $z$ satisfies the Variational Equations \eqref{eq.var}, 
after we interchange the $j$ and $i$ summation, we obtain that:

\begin{eqnarray*}
\b_{\ft}(z) & = & \sum_{i=0}^r \sum_{j=1}^{J}   
z_i \wedge  \left( (1-z^{A_j})^{v_i(\e_j A_j)} \right)\\
&=& \sum_{i=0}^r   z_i \wedge \prod_{j=1}^{J}
 (1-z^{A_j})^{v_i(\e_j A_j)} \\
& = & \sum_{i=0}^r   z_i \wedge (z^{-Q_i} \e^{-v_i(L)}) \\
& = & - \sum_{i=0}^r z_i \wedge z^{\frac{\pt Q}{\pt z_i}}
- \sum_{i=0}^r z_i^{v_i(L)} \wedge \e \\
& = & - \sum_{i=0}^r z_i \wedge z^{Q_i}
- z^L \wedge \e.
\end{eqnarray*}
On the other hand, we have:
$$
\sum_{i=0}^r z_i \wedge z^{Q_i} =
\sum_{i=0}^r z_i \wedge \prod_{j=0}^r z_j^{\frac{\pt Q}{\pt z_i \pt z_j}} 
=\sum_{i,j} \frac{\pt Q}{\pt z_i \pt z_j} ( z_i \wedge z_j ) 
$$
Since $Q$ is a symmetric bilinear form and $\wedge$ is 
skew-symmetric, it follows that
$$
\sum_{i=0}^r z_i \wedge z^{Q_i}=0.
$$
In addition, we claim that for $\e=\pm 1$, we have:

\begin{equation}
\lbl{eq.-1}
z \wedge \e=0 \in \BC^* \wedge \BC^*.
\end{equation}
Indeed, if $\e=1$, then 
$$
z \wedge 1=z \wedge(1 . 1)=z \wedge 1 +z \wedge 1.
$$ 
If $\e=-1$, then 
$$
z^2 \wedge (-1)=2 ( z \wedge (-1))=z \wedge ((-1)^2)=z \wedge 1 =0.
$$
Since $\BC^*=(\BC^*)^2$, the result follows.
Thus, 
$$
\nu(\b_{\ft}(z))=0 \in \BC^* \wedge \BC^*.
$$ 
This proves that $\b_{\ft}$ takes values in the Bloch group, and concludes
part (a) of Theorem \ref{thm.1}.

Part (b) follows from a stronger result; see part (c) of Theorem \ref{thm.1} 
below. The idea is that any analytic function is constant
on a connected component of its critical set. In our case, there exists 
a potential function whose points are the complex points of an affine 
variety defined over $\BQ$ by the Variational Equations. The 
complex points of an affine variety have finitely many connected components.
Finiteness of the image of $R \circ \b_{\ft}$ follows.
\end{proof}

\section{From \qterm s to the extended Bloch group}
\lbl{sec.extending}

In the remainder of the first part of the
paper, we will extend our results from the Bloch
group to the extended Bloch group. As a guide, given a \qterm\ $\ft$,
we will assign

\begin{itemize}
\item[(a)]
A potential function $V_{\ft}$; see Definition \ref{def.potential}.
\item[(b)]
A set of Logarithmic Variational Equations (see \eqref{eq.var}), 
which typically have
a finite set of solutions, given by algebraic numbers. The variational
equations may be identified with the critical points $\hatCP$
of the potential $V_{\ft}$; see Proposition \ref{prop.criticalV}. 
\item[(c)]
A map
$$
\hat{\b}_{\ft}: \hatCP \longto \hatBC
$$
where $\hatBC$ is the extended Bloch group; see Definition
\ref{def.exbloch} and Theorem \ref{thm.1}. 
When $\hat{\b}_{\ft}$ is composed with
a map $\hat{R}$ (given in Equation \eqref{eq.regulator}), 
it essentially
coincides with the evaluation of the potential on its critical values;
see Theorem \ref{thm.1}.
\item[(d)]
The image of the composition $e^{\hat{R}/(2 \pi i)} \circ \hat{\b}_{\ft}$ 
is a finite set of complex numbers, given by exponentials of periods; 
see Theorem \ref{thm.1}.
\end{itemize}

In the last part of the paper, which was a motivation all along, we mention
that quantum invariants of 3-dimensional knotted objects (or general
statistical-mechanical sums) give rise to \sqterm s, a fact that was 
initially observed in \cite{GL1}.

\subsection{The extended Bloch group}
\lbl{sub.extended}

Our aim is to associate elements of the extended Bloch group to every
\qterm. The extended Bloch group was introduced by Neumann in his
investigation of the Cheeger-Chern-Simons classes and hyperbolic 3-manifolds;
see \cite{Ne1}. The definition below is called the very-extended Bloch group
by Neumann, see also \cite{DZ}. 

There is a close relation between the Bloch group $\calB(F)$
and  $\Kind_3(F)$, (indecomposable $K$-theory of $F$); 
see \cite{B1,B2} and \cite{Su1,Su2}, as well as Section \ref{sub.cousins}
below.
Unfortunately, the relation among $\calB(F)$ and $\Kind_3(F)$ is 
known modulo torsion, and the torsion does not match. For $F=\BC$, this is 
exactly what the extended Bloch group
captures. The idea is that: 
\begin{center}
\fbox{torsion is encoded by the choice of the branch of the logarithms}.
\end{center}
More precisely, consider the doubly punctured plane
\begin{equation}
\lbl{eq.C**}
\BCs:=\BC\setminus\{0,1\}
\end{equation}
and let $\hat{\BC}$ denote the universal abelian cover of $\BCs$.
We can represent the Riemann surface of 
$\BCs$ as follows. Let $\BC_{\text{cut}}$ denote 
$\BCs$ cut open along each of the intervals $(-\infty,0)$
and $(1,\infty)$ so that each real number $r$ outside $[0,1]$ occurs
twice in $\BC_{\text{cut}}$. Let us denote the two occurrences of $r$ by 
$r+0i$ and $r-0i$ respectively. It is now easy to see that $\hat{\BC}$
is isomorphic to the surface obtained from $\BC_{\text{cut}} \times 2\BZ \times
2\BZ$ by the following identifications:

\begin{eqnarray*}
(x+0i;2p,2q) & \sim & (x-0i;2p+2,2q) \quad \text{for} \quad x \in (-\infty,0)
\\
(x+0i;2p,2q) & \sim & (x-0i;2p,2q+2) \quad \text{for} \quad x \in (1,\infty)
\end{eqnarray*}
This means that points in $\hat{\BC}$ are of the form $(z,p,q)$ with 
$z \in \BCs$ and $p,q$ even integers. Moreover, $\hat{\BC}$
is the Riemann surface of the function

\begin{equation*}
\BCs \longto \BC^2, \qquad z \mapsto (\Log(z), \Log(1-z)).
\end{equation*}
where $\Log$ denotes the {\em principal branch} of the logarithm.
Consider the set

\begin{eqnarray*}
\FT:=\{(x,y, \frac{y}{x}, \frac{1-x^{-1}}{1-y^{-1}},\frac{1-x}{1-y})\}
\subset (\BCs)^5
\end{eqnarray*}
of 5-tuples involved in the 5-term relation. Also, let

\begin{eqnarray*}
\FT_0:=\{(x_0,\dots,x_4) \in \FT | 0 < x_1 < x_0 < 1 \}
\end{eqnarray*}
and define $\hatFT \subset \hat{\BC}^5 $ to be the component of the
preimage of $\FT$ that contains all points $((x_0;0,0), \dots, (x_4;0,0))$
with $(x_0,\dots,x_4) \in \FT_0$.
See also \cite[Rem.2.1]{DZ}.

We can now define the extended Bloch group, following \cite[Def.1.5]{GZ}.

\begin{definition}
\lbl{def.exbloch}
\rm{(a)} The {\em extended pre-Bloch group} $\widehat{\calP(\BC)}$ is the 
abelian group generated by the symbols $[z;p,q]$ with $(z;p,q) \in \hat{\BC}$,
subject to the relation:
\begin{eqnarray}
\lbl{eq.ex5term}
\sum_{i=0}^4 (-1)^i [x_i;p_i,q_i]=0 \quad \text{for} \quad
((x_0;p_0,q_0),\dots, (x_4;p_4,q_4)) \in \hatFT.
\end{eqnarray}
\rm{(b)} The {\em extended Bloch group} $\hatBC$ is the 
kernel of the homomorphism
\begin{equation}
\lbl{eq.hatnu}
\hat{\nu}: \widehat{\calP(\BC)} \longto \BC \wedge \BC, \qquad
[z;p,q] \mapsto 
(\Log(z)+p \pi i) \wedge 
(-\Log(1-z)+q \pi i)
\end{equation}
\rm{(c)}
We will call the 2-term complex \eqref{eq.hatnu} the 
{\em extended Bloch-Suslin} complex.
\end{definition}

For a comparison between the extended Bloch-Suslin complex and the
Suslin complex, see Section \ref{sub.compareBS}.

\begin{remark}
\lbl{rem.4versions}
There are four versions of the extended Bloch group, depending on whether 
$(p,q) \in \BZ^2$ versus $(p,q) \in (2 \BZ)^2$, and whether we include the
transfer relation of \cite[Eqn.3]{Ne1}. For a comparison of Neumann's version
of the extended Bloch group with the above definition, see \cite[Rem.4.3]{GZ}.
\end{remark}

\subsection{The Rogers dilogarithm}
\lbl{sub.rogers}

The extended Bloch-Suslin complex \eqref{eq.hatnu} is very closely related
to the functional properties of a normalized form of the dilogarithm. 
In fact, it would be hard to motivate the extended Bloch-Suslin complex 
without knowing the Rogers dilogarithm and vice-versa.

Following \cite{Ne1,DZ}, the {\em Rogers
dilogarithm} is the following function defined on the open interval $(0,1)$:

\begin{equation}
\lbl{eq.rogers}
L(z)=\Li_2(z) + \frac{1}{2} \Log(z)\Log(1-z) -\frac{\pi^2}{6}
\end{equation}
where 
$$
\Li_2(z)=-\int_0^z \frac{\Log(1-t)}{t} dt=\sum_{n=1}^\infty \frac{z^n}{n^2}
$$ 
is the classical {\em dilogarithm function}. In \cite[Defn.1.5]{GZ},
Goette and Zickert extended $L(z)$ to a multivalued analytic function on 
$\hat{\BC}$ as follows (see \cite[Prop.2.5]{Ne1}):

\begin{equation}
\lbl{eq.hatrogers0}
\hat{R}:  \hat{\BC}  \longto \BC/\BZ(2) 
\end{equation}
\begin{eqnarray}
\lbl{eq.hatrogers}
\hat{R}(z;p,q) &:= & \Li_2(z)+\frac{1}{2}(\Log(z)+\pi i p)(\Log(1-z)+\pi i q)
-\frac{\pi^2}{6} \\
& = & L(z)+\frac{\pi i}{2} (q \Log(z)+p \Log(1-z))- \frac{\pi^2 p q}{2}.
\end{eqnarray}
where as usual, for a subgroup $K$ of $(\BC,+)$ and an integer $n \in \BZ$, 
we denote $K(n)=(2 \pi i)^nK \subset \BC$.

Let us comment a bit on the properties of the Rogers dilogarithm.

\begin{remark}
\lbl{rem.rogers0}
A computation involving the monodromy of the $\Li_2$ function shows that
the dilogarithm $\Li_2$ has an analytic extension:
\begin{equation}
\lbl{eq.extLi2}
\Li_2: \hat{\BC} \longto \BC/\BZ(2)
\end{equation}
See \cite[Prop.1]{Oe}. Compare also with \cite[Prop.2.5]{Ne1}. 
In \cite[Lem.2.2]{GZ} it is shown that $\hat{R}$ takes values in $\BC/\BZ(2)$
and that the Rogers dilogarithm is defined
on the extended pre-Bloch group; we will denote the extension by $\hat{R}$.
In other words, we have:
\begin{equation}
\lbl{eq.rogers3}
\hat{R}: \widehat{\calP(\BC)} \longto \BC/\BZ(2).
\end{equation}
In \cite{DZ} and \cite{GZ} it was shown that 
the Rogers dilogarithm coincides with the $\BC/\BZ(2)$-valued 
{\em Cheeger-Chern-Simons class} (denoted $\hat{C}_2$ in \cite{DZ}).
For a discussion on this matter, see \cite[Thm.4.1]{DZ}.
\end{remark}

\begin{remark}
\lbl{rem.rogers2}
It is natural to ask why to modify the dilogarithm as in Rogers
extension. The motivation
is that $\hat{L}$ (but {\em not} the dilogarithm function $\Li_2(z)$
nor the Bloch-Wigner dilogarithm) satisfies the extended 5-term relation 
$\hatFT$; see \cite[Sec.2]{DZ} and \cite[Lem.2.2]{GZ}. 
\end{remark}

\begin{remark}
\lbl{rem.rogers3} 
For the purposes of comparing the Rogers dilogarithm with other
special functions (such as the entropy function discussed in \cite{Ga1})
It is easy to see that the derivative of the Rogers dilogarithm
is given by:
\begin{equation}
\lbl{eq.derL}
L'(z)=-\frac{1}{2} \left( \frac{\log(1-z)}{z}+\frac{\log(z)}{1-z} \right).
\end{equation}
\end{remark}

\noindent
Thus, $\hat{R}$ descends to a map:

\begin{equation}
\lbl{eq.regulator}
\hat{R}: \hatBC \longto \BC/\BZ(2)
\end{equation}
which can be exponentiated to a map:

\begin{equation}
\lbl{eq.expregulator}
e^{\frac{1}{2 \pi i}\hat{R}}: \hatBC \longto \BC^*.
\end{equation}

The maps $R$ and $\hat{R}$ of the Bloch group and its extended version
are part of the following commutative diagram: 

$$
\divide\dgARROWLENGTH by2
\begin{diagram}
\node{\hatBC}
\arrow{e,t}{\hat{R}}
\arrow{s}
\node{\BC/\BZ(2)} 
\arrow{s,r}{i \Im } 
\arrow{e,t}{e^{\frac{\bullet}{2 \pi i}}}
\node{\BC^*}
\arrow{s,r}{| \bullet |}
\\
\node{\calB(\BC)}
\arrow{e,t}{R}
\node{i \BR}
\arrow{e,t}{e^{\frac{\bullet}{2 \pi i}}}
\node{\BR^+}
\end{diagram}
$$
For a proof of the commutativity of the left square, see \cite[Prop.4.6]{DZ}.
Let us end this section with a remark.

\begin{remark}
\lbl{rem.rogers4}
Even though the regulators $R$ and $\hat{R}$ are defined on the (extended)
pre-Bloch groups, the following diagram is {\em not} commutative:
$$
\divide\dgARROWLENGTH by2
\begin{diagram}
\node{\hatBC}
\arrow{e,t}{\hat{R}}
\arrow{s}
\node{\BC/(2 \pi^2 \BZ)} 
\arrow{s,r}{i \Im } 
\\
\node{\calB(\BC)}
\arrow{e,t}{R}
\node{i \BR}
\end{diagram}
$$
\end{remark}

\subsection{Two cousins of the extended Bloch group: 
$\Kind_3(\BC)$ and $\CH^2(\BC,3)$}
\lbl{sub.cousins}

The Bloch group $\calB(F)$ of a field has two cousins: $\Kind_3(F)$ and the 
higher Chow groups $\CH^2(F,3)$, also defined by Bloch; see \cite{B2}.
A spectral sequence argument (attributed to Bloch and Bloch-Lichtenbaum)
and some low-degree computations imply that for every infinite field $F$
we have an isomorphism:

\begin{equation}
\lbl{eq.KindCH}
\Kind_3(F) \cong \CH^2(F,3).
\end{equation}
For a detailed discussion, see \cite[Prop.5.5.20]{E-V}.
On the other hand, in \cite[Thm.5.2]{Su1} Suslin proves the existence of 
a short exact sequence:

\begin{equation}
\lbl{eq.suslin1}
0 \longto \text{Tor}(\mu_F,\mu_F\tilde{)} \longto \Kind_3(F) \longto \calB(F)
\longto 0.
\end{equation}
where $\text{Tor}(\mu_F,\mu_F\tilde{)}$ is the unique nontrivial extension of 
$\text{Tor}(\mu_F,\mu_F)$, where $\mu_F$ is the roots of unity of $F$. For
$F=\BC$, the above short exact sequence becomes:

\begin{equation}
\lbl{eq.suslin2}
0 \longto \BQ/\BZ \longto \Kind_3(\BC) \longto \calB(\BC)
\longto 0.
\end{equation}
Moreover, Suslin proves in \cite{Su1} that $\calB(\BC)$ is a $\BQ$-vector 
space. 

On the other hand, in \cite[Thm.3.12]{GZ} Goette and Zickert
prove that the extended Bloch group fits in a short exact sequence:

\begin{equation}
\lbl{eq.Ne}
0 \longto \BQ/\BZ \stackrel{\hat{\psi}}\longto \hatBC 
\longto \calB(\BC) \longto 0,
\end{equation}
for an explicit map $\hat{\psi}$. 
Equations 
\eqref{eq.KindCH}, \eqref{eq.suslin2}, \eqref{eq.Ne}, together with the fact 
that $\calB(\BC)$ is a $\BQ$-vector space, imply the following.

\begin{proposition}
\lbl{prop.compare}
There exist abstract isomorphisms:
\begin{equation}
\lbl{eq.abiso}
\hatBC \cong \Kind_3(\BC) \cong \CH^2(F,3).
\end{equation}
\end{proposition}

In addition there are maps (sometimes also known by the name
of {\em cycle maps} or {\em Abel-Jacobi maps}):
\begin{eqnarray*}
R': \Kind_3(\BC) & \longto & H^1_D(\Spec(\BC),\BZ(2))= \BC/\BZ(2) \\
R'': \CH^2(F,3)   & \longto & H^1_D(\Spec(\BC),\BZ(2))= \BC/\BZ(2)
\end{eqnarray*}
where $H_D$ denotes Deligne cohomology.
For an explicit formula for $R''$, see \cite[Sec.5.7]{KLM-S}.

Let us end this section with a problem and a question.
Recall that the extended Bloch group can be defined for any
subfield $F$ of the complex numbers, as discussed in \cite{Ne1}.

\begin{problem}
\lbl{prob.exbloch}
Define explicit isomorphisms in \eqref{eq.abiso} that commute with the
maps $R,R'$ and $R''$.
\end{problem}
We will come back to this problem in a forthcoming publication; \cite{GZ}.
For a careful relation between $\widehat{\calB(F)}$ and $\Kind_3(F)$
in the case of a number field $F$, see \cite{Zi}.

\subsection{\qterm s and potential functions}
\lbl{sub.potential}

In this section, we will assign elements of $\hatBC$
to a \qterm\ $\ft$; see Theorem \ref{thm.1} below. 

\begin{definition}
\lbl{def.Phi}
Consider the following multivalued function on $\BCs$:
\begin{equation}
\lbl{eq.Phi}
\Phi: \hat{\BC} \longto \BC/\BZ(1), \qquad
\Phi(x)=\frac{1}{2 \pi i}\left(\frac{\pi^2}{6}-\Li_2(x) \right).
\end{equation}
\end{definition}

\noindent
Remark \ref{rem.rogers0} shows that indeed $\Phi$ takes values in 
$\BC/\BZ(1)$. The function $x \mapsto \Phi(e^{2 \pi i x})$ appears in work of 
Kashaev on the Volume Conjecture; see \cite{Ka}. The relevance of the 
special function $\Phi$ is that it describes the growth rate of
the $q$-factorials at complex roots of unity; see Lemma
\ref{lem.qfactorials} in Section \ref{sub.motivation}.
The next definition associates a potential function to a \qterm.

\begin{definition}
\lbl{def.potential}
Given a \qterm\ $\ft$ as in \eqref{eq.qterm}, 
let us define its potential function $V_{\ft}$ by:
\begin{equation}
\lbl{eq.potential}
V_{\ft}(z)=\frac{1}{2 \pi i} Q(\Log( z)) + \frac{1}{2 \pi i} \Log \e \cdot 
\Log(z^L) +\sum_{j=1}^{J} \e_j \Phi(z^{A_j})
\end{equation}
where $z=(z_0,z_1,\dots,z_r)$ and $\Log(z)=(\Log(z_0),\dots,\Log(z_r))$.
 Let $\hatCP$ and 
denote the {\em critical points} of the potential $V_{\ft}$.
\end{definition}

Since $V_{\ft}(z)$ is a multivalued function, let us explain its domain.
Given a \qterm\ $\ft$ as in \eqref{eq.qterm}, let 
denote 

\begin{equation}
\lbl{eq.Hft}
H_{\ft}=\{z \sub (\BC^*)^{r+1} \, | z^L
\prod_{j=1}^J (1-z^{A_j}) \prod_{i=0}^r (1-z_i)=0
\}.
\end{equation}
Let $\calD_{\ft}$ denote the universal abelian cover of $(\BC^*)^{r+1}\setminus
H_{\ft}$. 
Observe that for every $i=0,\dots,r$, $j=1,\dots,J$
there are well-defined analytic maps:

\begin{equation}
\lbl{eq.proj} 
\pi_i, \pi_{A_j}, \pi_L : (\BC^*)^{r+1}\setminus H_{\ft} \to \BCs
\end{equation}
given by $\pi_i(z)=z_i$, $\pi_{A_j}(z)=z^{A_j}$ and $\pi_L(z)=z^L$; using
the notation of \eqref{eq.zA}. Lifting them
to the universal abelian cover, gives rise to analytic maps:

\begin{equation}
\lbl{eq.hatproj} 
\hat{\pi}_i,\hat{\pi}_{A_j},\hat{\pi}_L : 
\calD_{\ft} \to \hat{\BC}.
\end{equation}
We denote the image of $z \in \calD_{\ft}$ under these maps by $z_i$,
$z^{A_j}$ and $z^L$ respectively.
With these conventions we have:

\begin{lemma}
\lbl{lem.domainVt}
Equation \eqref{eq.potential} defines an analytic function:
\begin{equation}
\lbl{eq.domainVt}
V_{\ft}: \calD_{\ft} \longto \BC/\BZ(1).
\end{equation}
\end{lemma}

The next proposition describes the critical points and the critical values
of the potential function.

\begin{proposition}
\lbl{prop.criticalV}
\rm{(a)} 
The critical points $z$ of $V_{\ft}$ are the solutions
to the following system of {\em Logarithmic Variational Equations}:
\begin{equation}
\lbl{eq.varlog}
\sum_{j=1}^{J}  \e_j v_i(A_j) \Log(1-z^{A_j}) 
+ 
\frac{\pt Q}{\pt z_i}(\Log(z)) + 
\Log \e \cdot v_i(L)=0
\end{equation}
for $i=0,\dots,r$.
\newline
\rm{(b)} There is a map:
\begin{equation}
\lbl{eq.exp}
\hatCP \longto X_{\ft}, \qquad z \mapsto (\pi(z_1),\dots,\pi(z_r))
\end{equation}
where $\pi: \hat{\BC} \to \BCs$ is the projection map.
\end{proposition}

For $A=A_j$, ($j=1,\dots,J$) or $A=L$, 
let $p_{z,A} \in 2 \BZ$ (or simply, $p_A$ if
$z$ is clear) be defined so that we have:

\begin{equation}
\lbl{eq.LogA}
\Log(z^A)=\sum_{i=0}^r v_i(A)\Log(z_i) - 
p_{z,A} \pi i.
\end{equation}

Given $w=(x;p_0,q_0) \in \hat{\BC}$ and even integers $p,q \in 2 \BZ$
let us denote

\begin{equation}
\lbl{eq.pqw}
T_1^p T_0^q (w):=(x;p_0+p,q_0+q) \in \hat{\BC}
\end{equation}
In other words, $T_1$ and $T_0$ are generators for the deck transformations
of the $\BZ^2$-cover:

\begin{equation}
\lbl{eq.pi}
\pi: \hat{\BC} \longto \BCs.
\end{equation}

We need one more piece of notation: for $z \in \calD_{\ft}$, consider the
corresponding elements $z^L$ and $z_i$ of $\hat{\BC}$ for $i=0,\dots,r$. We
denote by $z^{-L/2} \in \hat{\BC}$ the unique element of $\hat{\BC}$
that solves the equation:

\begin{equation}
\lbl{eq.sqrt}
-\Log(z^L)-\Log(z^{-L/2})+\frac{1}{2} \sum_{i=0}^r v_i(L) \Log(z_i)=0.
\end{equation}

\begin{definition}
\lbl{def.beta}
With the above conventions, consider the map
\begin{equation}
\lbl{eq.hatbeta}
\hat{\b}_{\ft}: \hatCP \longto \widehat{\calP(\BC)}
\end{equation}
given by:
\begin{equation}
\lbl{eq.hatbetab}
w \mapsto \hat{\b}_{\ft}(w):=[z^{-L/2};0,2 \frac{\Log\e}{\pi i}]
-[z^{-L/2};0,0]+
\sum_{j=1}^{J} \e_j [T_1^{p_{z,A_j}}(z^{A_j})] 
\end{equation}
\end{definition}

Our next definition assigns a numerical invariant to a special
term $\ft$.

\begin{definition}
\lbl{def.hatSE}
For a \qterm\ $\ft$, let
\begin{equation}
\lbl{eq.hatS}
\CV_{\ft}=\{e^{-V_{\ft}(z)} \,\, | z \,\, 
\text{satisfies \eqref{eq.varlog}} \}.
\end{equation}
\end{definition}

Our next main theorem links the exponential of the critical values of $V_{\ft}$
with the values of $\hat{\b}_{\ft}$. Let $\calP$ denote
the set of {\em periods} in the sense of Kontsevich-Zagier, 
\cite{KZ}. $\calP$ is a countable subset of the complex numbers.

\begin{theorem}
\lbl{thm.1}
\rm{(a)} The map $\hat{\b}_{\ft}$ descends to a map 
\begin{equation}
\lbl{eq.hatbeta2}
\hat{\b}_{\ft}: \hatCP \longto \hatBC
\end{equation}
which we denote by the same name.
\newline
\rm{(b)} We have a commutative diagram:
$$
\divide\dgARROWLENGTH by2
\begin{diagram}
\node{\hatCP}
\arrow[2]{e,t}{\hat{\b}_{\ft}}\arrow{se,r}{e^{-V_{\ft}}} 
\node[2]{\hatBC}
\arrow{sw,r}{e^{\frac{1}{2\pi i}\hat{R}}} \\
\node[2]{\BC^*}
\end{diagram}
$$
\newline
\rm{(c)} 
$\CV_{\ft}$ is a finite subset of $e^{\calP}$. Thus, we get a map:
\begin{equation}
\lbl{eq.finite}
\CV: \text{\qterm s} \longto \text{Finite subsets of $e^{\calP}$}.
\end{equation}
\end{theorem}

\begin{example}
\lbl{ex.continue2}
Let us continue with the Example \ref{sub.example}. When $a=-1$, $b=2$
and $\e=-1$, the Variational Equations \eqref{eq.abe2} have solutions
$$
X_{\ft}=\{z^{\pm 1}\, | \, z=e^{2 \pi i/6}\}.
$$
In that case, the corresponding elements of the extended Bloch group are

\begin{equation*}
2[z^{\pm 1};0,0]+[z^{-1/2};0,2]-[z^{-1/2};0,0],
\end{equation*}
and their values under the map $\hat{R}$ are
given by:

\begin{equation*}
0 \pm  i \, 2.0298832128193074\dots 
\end{equation*}
whose imaginary part equals to the {\em Volume} $\text{Vol}(4_1)$
of the $4_1$ knot; see \cite{Th}.
Moreover,

\begin{eqnarray*}
\CV_{\ft} &=& \{e^{\pm \frac{1}{2 \pi}\text{Vol}(4_1)} \} \\
&=& \{ 0.7239261119 \dots, 1.3813564445 \dots \}.
\end{eqnarray*}
\end{example}

\section{Proof of Theorem \ref{thm.1}}
\lbl{sec.proofs}

In this section we will give the proofs of Proposition 
\ref{prop.criticalV} and Theorem \ref{thm.1}.

\subsection{Proof of Proposition \ref{prop.criticalV}}
\lbl{sub.prop.criticalV}

It suffices to show that for every $i=0,\dots,r$, 
$2 \pi \sqrt{-1} z_i \frac{\pt V_{\ft}}{\pt z_i}$
is given by the left hand side of Equation \eqref{eq.varlog}.
This follows easily from the definition of $V_{\ft}$ and the elementary
computation:

\begin{equation}
\lbl{eq.Phi'}
\Phi'(x)= \frac{1}{2 \pi i} \frac{\Log(1-x)}{x}.
\end{equation}

\subsection{Proof of Theorem \ref{thm.1}}
\lbl{sub.thm.1}

The proof is similar to the proof of Theorem \ref{thm.0}, once we keep
track of the branches of the logarithms. 
Fix a \qterm\ $\ft$ and consider the map given by \eqref{eq.hatbetab}.
Without loss of generality,
we assume that 
$$
Q(k)=\frac{1}{2}\sum_{i,j=0}^r k_i k_j
$$
Let us fix $z=(z_0,\dots,z_r)$ that satisfies the Variational Equations
\eqref{eq.var}. 
We will show that $\hat\nu(\hat{\b}_{\ft}(z))=0 \in \BC \wedge \BC$, where 
$\hat\nu$ is given by \eqref{eq.hatnu}.  
We have
$$
\hat{\b}_{\ft}(z)=\hat{\b}_{1,\ft}(z)+\hat{\b}_{2,\ft}(z)
$$
where
\begin{eqnarray*}
\hat{\b}_{1,\ft}(z) &=&
\sum_{j=1}^{J} \e_j [T_1^{p_{z,A_j}}(z^{A_j})] \\
\hat{\b}_{2,\ft}(w) &=&
[z^{-L/2};0,2 \frac{\Log\e}{\pi i}]-[z^{-L/2};0,0].
\end{eqnarray*}
It follows that

\begin{eqnarray*}
\hat\nu(\hat{\b}_{\ft}(z)) & = & \hat\nu(\hat{\b}_{1,\ft}(z))
+ \hat\nu(\hat{\b}_{2,\ft}(z)) 
\end{eqnarray*}
where
\begin{eqnarray*}
\hat\nu(\hat{\b}_{1,\ft}(z)) & = &
\sum_{j=1}^{J} \e_j (\Log( z^{A_j})+p_{z,A_j} \pi i) 
\wedge (-\Log(1-z^{A_j})) \\
\hat\nu(\hat{\b}_{2,\ft}(z)) &=& 
\Log(z^{-L/2}) \wedge  2 \Log\e \\
&=&  2 \Log(z^{-L/2}) \wedge \Log\e.
\end{eqnarray*}

On the other hand, using Equation \eqref{eq.LogA}, 
we have:

\begin{eqnarray*}
(\Log( z^{A})+p_{z,A} \pi i) \wedge \Log(1-z^{A}) 
&=& \sum_{i=0}^r v_i(A) ( \Log(z_i) \wedge \Log(1-z^{A})) \\
&=& \sum_{i=0}^r \Log(z_i) \wedge ( v_i(A) \Log(1-z^{A})).
\end{eqnarray*}

\noindent
Since $z$ satisfies the Logarithmic Variational Equations 
\eqref{eq.varlog}, after we interchange the $j$ and $i$ summation, 
we obtain that:

\begin{eqnarray*}
\hat\nu(\hat{\b}_{1,\ft}(z)) & = & \sum_{i=0}^r \sum_{j=1}^{J}   
\Log(z_i) \wedge (-\e_j v_i(A_j) \Log(1-z^{A_j})) \\
&=& \sum_{i=0}^r  \Log( z_i) \wedge ( \sum_{j=1}^{J} 
-\e_j v_i(A_j) \Log(1-z^{A_j}) ) \\
& = & \sum_{i=0}^r  \Log( z_i) \wedge (
\frac{\pt Q}{\pt z_i}(\Log(z)) + \Log \e \cdot v_i(L)). 
\end{eqnarray*}
Since $Q$ is an integral symmetric bilinear form and $\wedge$ is 
skew-symmetric, it follows that

\begin{eqnarray*}
\sum_{i=0}^r  \Log( z_i) \wedge (
 \frac{\pt Q}{\pt z_i}(\Log(z)) &=&
0.
\end{eqnarray*}
Moreover,

\begin{eqnarray*}
\sum_{i=0}^r  \Log( z_i) \wedge (
\Log \e \cdot v_i(L)) &=& 
\sum_{i=0}^r v_i(L) \Log(z_i) \wedge \Log \e \\
&=& ( \Log(z^L)+ p_{z,L} \pi i) \wedge \Log \e \\
&=&  \Log(z^L) \wedge \Log \e.
\end{eqnarray*}

Thus,
$$
\hat\nu(\hat{\b}_{1,\ft}(z)) = \Log(z^L) \wedge \Log \e
$$
which implies that
$$
\hat\nu(\hat{\b}_{\ft}(z))=(2 \Log(z^{-L/2})+\Log(z^L)) \wedge \Log \e.
$$ 
On the other hand, reducing Equation 
\eqref{eq.sqrt} modulo $\pi i \BZ$, and using \eqref{eq.LogA}
we have:

\begin{eqnarray*}
0 &=& -\Log(z^L)-\Log(z^{-L/2})+\frac{1}{2} \sum_{i=0}^r v_i(L) \Log(z_i)
\\
&=& -\Log(z^L)-\Log(z^{-L/2})+\frac{1}{2} (\Log(z^L) +\pi i  p_{z,L})
\\
&=&
-\frac{1}{2} (\Log(z^L)+2 \Log(z^{-L/2})) + \frac{\pi i  p_{z,L}}{2} 
\\
&=&
-\frac{1}{2} (\Log(z^L)+2 \Log(z^{-L/2})).
\end{eqnarray*}
In other words, we have:

\begin{equation}
\lbl{eq.sqrt2}
\Log(z^L)+2 \Log(z^{-L/2}) \in \frac{\pi i}{2} \BZ.
\end{equation}

Since $\Log\e \in \pi i \BZ$, it follows that $\hat\nu(\hat{\b}_{\ft}(z))=0$
and concludes part (a) of Theorem \ref{thm.1}.

For part (b), observe first that for every $w \in \hat{\BC}$ and 
every even integers $p$ and $q$ we have:

\begin{equation*}
\hat{R}([T_1^p T_0^q (w)])=\hat{R}([w])+ \frac{\pi i}{2} 
(q \Log(z)+p \Log(1-z)).
\end{equation*}

Moreover, by the definition of $\Phi$ and by Equation \eqref{eq.LogA},
for any integral linear form $A=A_j$ for 
$j=1,\dots, J$ or $A=L$, we have:

\begin{eqnarray*}
-\Phi(z^A) &=& \frac{1}{2 \pi i} \left(\hat{R}([T_1^{p_{z,A}} (z^{A})]) 
-\frac{1}{2} \Log(z^{A}) \Log(1-z^{A}) -\frac{\pi i}{2} p_{z,A} \Log(1-z^A)
\right).
\end{eqnarray*}
Thus, using the definition of the potential function from Equation 
\eqref{eq.potential}, we obtain that:  

\begin{equation*}
-V_{\ft}(z) =   \frac{1}{2 \pi i} R(\hat{\b}_{\ft}(z)) + 
T_1 +T_2
\end{equation*}
where
\begin{eqnarray*}
T_1 &=& \frac{1}{2\pi i}\left(
-\frac{1}{2} \sum_j \e_j \Log(z^{A_j})\Log(1-z^{A_j}) -\frac{\pi i}{2} 
p_{z,A_j} \Log(1-z^{A_j}) \right) \\
T_2 &=& -\frac{1}{2 \pi i} Q(\Log(z)) -\frac{1}{2 \pi i}\Log\e \cdot 
\Log(z^L) 
-\frac{1}{2 \pi i} ( \hat{R}([z^{-L/2};0, \frac{2 \Log\e}{\pi i}]-
[z^{-L/2};0,0]) \\
&=& -\frac{1}{2 \pi i} Q(\Log(z)) -\frac{1}{2 \pi i}\Log\e \cdot \Log(z^L) 
-\frac{\Log\e}{2 \pi i} \Log(z^{-L/2}). 
\end{eqnarray*} 

Using \eqref{eq.LogA}, and interchanging $i$ and $j$ summation, and
using the Logarithmic Variational Equations \eqref{eq.varlog}, 
it follows that: 

\begin{eqnarray*}
\sum_{j=0}^J \e_j \Log(z^{A_j}) \Log(1-z^{A_j}) & = &
\sum_{j=0}^J
(\sum_{i=0}^r \e_j v_i(A_j) \Log(z_i) -p_{z,A_j} \pi i) \Log(1-z^{A_j}) \\
& = & \sum_{i=0}^r \Log(z_i) \sum_{j=0}^J  \e_j v_i(A_j) \Log(1-z^{A_j}) 
-\sum_{j=0}^J \e_j p_{z,A_j} \pi i \Log(1-z^{A_j}) \\
&=&
\sum_{i=0}^r \Log(z_i) (- \frac{\pt Q}{\pt z_i} -\Log\e \cdot v_i(L))
-\sum_{j=0}^J \e_j p_{z,A_j} \pi i \Log(1-z^{A_j}).
\end{eqnarray*}

\noindent
Thus,

\begin{eqnarray*}
T_1 &=& \frac{1}{4 \pi i} \sum_{i=0}^r \Log(z_i) \frac{\pt Q}{\pt z_i}(\Log(z))
+\frac{\Log\e}{4 \pi i} \sum_{i=0}^r v_i(L) \Log(z_i).
\end{eqnarray*}

\noindent
Since $Q$ is a symmetric bilinear form, it follows that 
$$
\frac{1}{2}
\sum_{i=0}^r \Log(z_i) \frac{\pt Q}{\pt z_i}(\Log(z)) - Q(\Log(z))=0.
$$
This, together with Equation \eqref{eq.sqrt} implies that $T_1+T_2=0$. 
Exponentiating, it follows that
$$
e^{-V_{\ft}(z)}=e^{\frac{1}{2 \pi i} \hat{R}(\hat{\b}_{\ft}(z))} \in \BC^*
$$
which concludes part (b) of Theorem \ref{thm.1}.

For part (c), since an analytic function is constant on a connected 
component of its sets of critical points, and since the set of critical
points $\hatCP$ is the set of complex points of an affine variety 
defined over $\BQ$, it follows that $\hatCP$ has finitely many connected
components. Thus, $\CV_{\ft}$ is a finite subset of $\BC$. Moreover,
since the value of the map \eqref{eq.regulator} 
at a point $[z;p,q]$ with $z \in \overline\BQ$
is a period, in the sense of Kontsevich-Zagier, and since every connected
component of $\hatCP$ has a point $w$ so that $z=e^{2 \pi i w} \in 
(\overline\BQ)^{r+1}$ (due to the Variational Equations \eqref{eq.var}), it
follows that $\CV_{\ft}$ is a subset of $e^{\calP}$. This concludes
the proof of Theorem \ref{thm.1}.
\qed

\section{\Sqterm s, generating series and singularities}
\lbl{sec.gseries}

The second part of the paper assigns to germs of analytic functions
$\Lnp_{\ft}(z)$ and $\Lp_{\ft}(z)$ to a \sqterm, and formulates
a conjecture regarding their analytic continuation in the complex plane
minus a set of singularities related to the image of the composition
$\hat{R} \circ \hat{\b}_{\ft}$.

\subsection{What is a \sqterm ?}
\lbl{sub.sqterm}

In this section we introduce the notion of a \sqterm. 
Of course, every \sqterm\ is a \qterm . Examples of \sqterm s that come
naturally from Quantum Topology are discussed in Section \ref{sub.exspecial}
below.

\begin{definition}
\lbl{def.sqterm}
A {\em special $q$-hypergeometric term} $\ft$ (in short, 
{\em \sqterm}) in the $r+1$ variables $k=(k_0,k_1,\dots,k_r) \in \BN^r$
is a list that consists of
\begin{itemize}
\item
An integral symmetric quadratic form $Q(k)$ in $k$,
\item
An integral linear form $L$ in $k$ and a vector $\e=(\e_0,\dots,\e_r)$
with $\e_i=\pm 1$ for all $i$,
\item
Integral linear forms $B_j,C_j,D_j,E_j$ in $k$ for $j=1,\dots,J$
\end{itemize}
The list $\ft$ satisfies the following condition. Its {\em Newton polytope}
$P_{\ft}$ defined by
\begin{equation}
\lbl{eq.Pft}
P_{\ft}=\{ w \in \BR^r \,|  \, B_j(1,w) \geq C_j(1,w) \geq 0, \,\,
\,\,
D_j(1,w) \geq E_j(1,w) \geq 0, \,\, j=1,\dots, J\} 
\end{equation}
is a nonempty rational compact polytope of $\BR^r$. 
\end{definition}
This accurate but rather obscure definition is motivated from the
fact that 
a \sqterm\ $\ft$ gives rise to an expression of the form
\begin{equation}
\lbl{eq.sqterm}
\ft_{k}(q)=q^{Q(k)} \e^{L(k)} \prod_{j=1}^J 
\qbinom{B_j(k)}{C_j(k)} \frac{(q)_{D_j(k)}}{(q)_{E_j(k)}} \in \BZ[q^{\pm 1}]
\end{equation}
valid for $k \in \BN^{r+1}$ such that $B_j(k) \geq C_j(k) \geq 0$ and $D_j(k)
\geq E_j(k) \geq 0$ for all $j=1,\dots,J$.

\subsection{From \sqterm s to generating series}
\lbl{sub.sgen}

We now discuss how a \sqterm\ $\ft$ gives rise to a sequence of Laurent
polynomials $(a_{\ft,n}(q))$ and to a generating series $G_{\ft}(z)$.
Given a \sqterm\ $\ft$ in $r+1$ variables $k=(k_0,\dots,k_r)$, it will
be convenient to single out the first variable $k_0$, and denote it by $n$
as follows:
\begin{equation}
\lbl{eq.k0}
n=k_0, \qquad k'=(k_1,\dots,k_r)
\end{equation}
With the above convention, we have $k=(n,k')$. 

\begin{definition}
\lbl{def.aftn}
A \sqterm\ $\ft$ gives rise
to a sequence of Laurent polynomials $(a_{\ft,n}(q))$ as follows:
\begin{equation}
\lbl{eq.an}
a_{\ft,n}(q)=\sum_{k' \in n P_{\ft} \cap \BN^r} \ft_{n,k'}(q) \in \BZ[q^{\pm 1}]
\end{equation}
where the summation is over the finite set $n P_{\ft} \BN^r$ of lattice points
of the translated Newton polytope of $\ft$.
\end{definition}

\begin{definition}
\lbl{def.Gft}
A \sqterm\ $\ft$ gives rise to a power series $\Lnp_{\ft}(z)$ defined by:
\begin{equation}
\lbl{eq.ahabiro}
\Lnp_{\ft}(z)=\sum_{n=0}^\infty a_{\ft,n}(e^{\frac{2 \pi i}{n}})z^n. 
\end{equation}
\end{definition}

The next lemma proves that $\Lnp_{\ft}(z)$ is an analytic function at $z=0$.

\begin{lemma}
\lbl{lem.an1}
For every \sqterm\ $\ft$, $\Lnp_{\ft}(z)$ is analytic at $z=0$.
\end{lemma}

\begin{proof}
Fix a \sqterm\ $\ft$ as Definition \eqref{eq.sqterm}. Let $||f(q)||_1$
denote the sum of the absolute values of the coefficients of a Laurent
polynomial $f(q) \in \BQ[q^{\pm 1}]$. It suffices to show that there exists
$C>0$ so that
 
\begin{equation}
\lbl{eq.norm1}
||a_{\ft,n}(q)||_1 \leq C^n
\end{equation}
for all $n$. In that case, since $|e^{2 \pi i/n}|=1$, it follows that
$$
|a_{\ft,n}(e^{\frac{2 \pi i}{n}})| \leq C^n
$$
for all $n$, thus $\Lnp_{\ft}(z)$ is analytic for $|z|<1/C$. Now,
recall that for natural numbers $a,b$ with $a\geq b \geq 0$ we have:
$$
\qbinom{a}{b} \in \BN[q^{\pm 1}], \qquad ||\qbinom{a}{b}||_1=\binom{a}{b}
\leq 2^a.
$$
(see for example, \cite{St}). In addition for natural numbers 
$c,d$ with $c \geq d \geq 0$ we have:
$$
\frac{(q)_c}{(q)_d}=\prod_{j=d+1}^c (1-q^j),
\qquad
||\frac{(q)_c}{(q)_d}||_1 \leq 2^{c-d} \leq 2^c.
$$
Equation \eqref{eq.an} implies that
$$
||a_{\ft,n}(q)||_1 \leq \sum_{k' \in n P_{\ft} \cap \BN^r} ||\ft_{n,k'}(q)||_1
$$
Since the number of terms in the above sum is bounded by a polynomial
function of $n$, and the summand is bounded by an exponential function
of $n$, Equation \eqref{eq.norm1} follows.
\end{proof}

\subsection{Examples of \sqterm s from Quantum Topology}
\lbl{sub.exspecial}

Quantum Topology gives a plethora of \sqterm s $\ft$ to knotted 3-dimensional
objects whose corresponding sequence $\a_{\ft,n}(q)$ depends on the knotted
object itself. A concrete example is given in \cite[Sec.3]{GL1}. 
To state it, let $J_{K,n}(q) \in \BZ[q^{\pm 1}]$ denote the {\em colored
Jones polynomial} of a {\em knot} $K$, colored by the $n$ dimensional
irreducible representation of $\mathfrak{sl}_2$ and normalized to equal
to $1$ at the unknot; see \cite{Tu}. 
In \cite[Lem.3.2]{GL1} the following is shown. 

\begin{lemma}
\lbl{lem.braid}
Let $\b$ denote a braid whose closure is a knot $K$. Then, there exists a 
\sqterm\ $\ft_{\b}$ such that
\begin{equation}
J_{K,n}(q)=a_{\ft_{\b},n}(q)
\end{equation}
for all $n \in \BN$. It follows that $a_{\ft_{\b},n}(e^{\frac{2 \pi i}{n}})$ is the
$n$-th {\em Kashaev invariant} of $K$; \cite{Ka}.
\end{lemma}
Please observe that the \sqterm\ $\ft_{\b}$ depends on the braid $\b$,
and that a fixed knot $K$ can always be obtained as the closure of infinitely
many braids $\b$. Nonetheless, for all such braids $\b$, the sequence
of polynomials $(a_{\ft_{\b},n}(q))$ depend only on the knot $K$.

\begin{remark}
\lbl{rem.computer}
A computer implementation of Lemma \ref{lem.braid} is available
from \cite{B-N}. It uses as input a braid word in the standard generators
of the braid group, and outputs the expression \eqref{eq.sqterm}
of the corresponding \sqterm .
\end{remark}

Quantum Topology constructs 
many more examples of \sqterm s, that depend on a pair $(K,\mathfrak{g})$ 
of a knot $K$ and a simple Lie algebra $\mathfrak{g}$, or to a pair
$(M,\mathfrak{g}$) of a closed 3-manifold $M$ with the integer
homology of $S^3$ and a simple Lie algebra $\mathfrak{g}$.

\subsection{An ansatz for the singularities of $\Lnp_{\ft}(z)$}
\lbl{sub.ansatz}

In this section we connect the two different parts of the paper. Namely,
a \sqterm\ $\ft$ gives rise to 
\begin{itemize}
\item[(a)]
a finite set $\CV_{\ft} \subset \BC^*$ 
from Definition \ref{def.hatSE},
\item[(b)] 
a generating series $\Lnp_{\ft}(z)$ from Definition \ref{def.Gft}.
\end{itemize}

\begin{conjecture}
\lbl{conj.1}
For every \sqterm\ $\ft$, the germ $\Lnp_{\ft}(z)$ has an analytic 
continuation as a multivalued function in $\BC\setminus (\CV_{\ft} \cup \{0\})$.
Moreover, the local monodromy of  $\Lnp_{\ft}(z)$ is {\em quasi-unipotent},
i.e., its eigenvalues are complex roots of unity.
\end{conjecture}

\begin{example}
\lbl{ex.continue3}
If the \sqterm\ is given by \eqref{eq.abe1}, then Conjecture 1
is known; see \cite{CG}. With the notation of \eqref{eq.abe1},
the case of $(a,b,\e)=(-1,2,-1)$ coincides with the Kashaev
invariant of the $4_1$ knot.
\end{example} 

In a sequel to this paper \cite{Ga3}, we 
discuss in detail examples of Conjecture
\ref{conj.1} for \sqterm s that comes from Quantum Topology.

\section{For completeness}
\lbl{sec.completeness}

\subsection{Motivation for the special function $\Phi$ and the potential 
function}
\lbl{sub.motivation}

Recall the special function $\Phi$ given by Equation \eqref{eq.Phi}.
The next lemma is our motivation for introducing $\Phi$. Kashaev informs
us that this computation was well-known to Faddeev, and was a starting
point in the theory of $q$-dilogarithm function; see \cite{FK}.

\begin{lemma}
\lbl{lem.qfactorials}
For every $\a \in (0,1)$ we have:
\begin{equation}
\lbl{eq.qfactorials}
\prod_{k=1}^{[\a N]} (1-e^{\frac{2 \pi i k}{N}}) = e^{N \Phi(e^{2 \pi i \a})+ 
O\left(\frac{\log N}{N} \right)}.
\end{equation}
\end{lemma}

\begin{proof}
The proof is similar to the proof of \cite[Prop.8.2]{GL2}, and follows 
from applying the {\em Euler-MacLaurin summation formula} 
\begin{eqnarray*}
\log(\prod_{k=1}^{[\a N]} (1-e^{\frac{2 \pi i k}{N}})) & = &
\sum_{k=1}^{[\a N]} \log(1-e^{\frac{2 \pi i k}{N}}) \\
&= & \a N \int_0^1 \log(1-e^{2 \pi i \a x}) dx + O\left(\frac{\log N}{N} \right),
\end{eqnarray*} 
together with the fact that:
\begin{eqnarray*}
\int_0^1 \log(1-e^{2 \pi i ax}) dx =
\frac{1}{2 \pi i \a} \left(\frac{\pi^2}{6}
- \Li_2(e^{2 \pi i \a}) \right).
\end{eqnarray*}
\end{proof}

Fix a \sqterm\ $\ft_{k}(q)$ where $k=(k_1,\dots,k_r)$ and a positive
natural number $N \in \BN$. Fix also $w=(w_0,\dots,w_r)$ where $w_i \in (0,1)$
for $i=0,\dots,r$. Let us abbreviate $([w_0 N], [w_1 N], \dots, [w_r N])$ by 
$[w N]$. Lemma \ref{lem.qfactorials} implies the following.

\begin{lemma}
\lbl{lem.motpot}
With the above assumptions, we have:
\begin{equation}
\lbl{eq.motpot}
\log \ft_{[w N]} =  e^{N V_{\ft}(e^{2 \pi i w})+
O\left(\frac{\log N}{N} \right)}.
\end{equation}
\end{lemma}

This motivates our definition of the potential function.

\subsection{Laplace's method for a \qterm}
\lbl{sub.laplace}

There is an alternative way to derive the Variational Equations 
\eqref{eq.var} from a \qterm\ $\ft$.

Since $\ft_{k}$ is $q$-hypergeometric, and $k=(k_1,\dots,k_r)$, 
it follows that for every $i=0,\dots,r$ we have

\begin{equation*}
R_i(z_0,\dots,z_r,q):=\frac{\ft_{n,k_1,\dots,k_i+1,\dots,k_r}(q)}{
\ft_{n,k_1,\dots,k_r}(q)} \in \BQ(z_0,\dots,z_r,q)
\end{equation*}
where $z_i=q^{k_i}$ for $i=1,\dots,r$ and $z_0=q^n$. It is easy to see that
the system of equations:

\begin{eqnarray*}
R_i(z_0,z_1,\dots,z_r)=1, \qquad i=0,\dots,r.
\end{eqnarray*}
is identical to the system \eqref{eq.var} of variational equations.
In discrete math, the above system is known as {\em Laplace's method}.

\subsection{A comparison between the Bloch group and its extended version}
\lbl{sub.compareBS}

A comparison between the extended Bloch-Suslin complex and the
Suslin complex is summarized in the following diagram with short exact
rows and columns. The diagram is taken by combining \cite[Sec.3]{GZ} with
\cite[Thm.7.7]{Ne1}, and using the map $\hat{\chi}$ from \cite[Eqn.3.11]{GZ}.

$$
\divide\dgARROWLENGTH by2
\begin{diagram}
\node[2]{0}
\arrow{s}
\node{0}
\arrow{s}
\node{0}
\arrow{s}
\\
\node{0}
\arrow{e}
\node{\BQ/\BZ}
\arrow{e,t}{e^{2 \pi i \bullet}}
\arrow{s,r}{\hat{\psi}}
\node{\BC^*}
\arrow{e}
\arrow{s,r}{\hat{\chi}}
\node{\BC/(\BQ/\BZ)}
\arrow{s,r}{\tau} 
\arrow{e}
\node{0}
\arrow{s}
\\  
\node{0}
\arrow{e}
\node{\hatBC}
\arrow{e}
\arrow{s}
\node{\widehat{\calP(\BC)}}
\arrow{e,t}{\hat{\nu}}
\arrow{s}
\node{\BC \wedge \BC}
\arrow{s,r}{-\exp \wedge \exp} 
\arrow{e}
\node{K_2(\BC)}
\arrow{s}
\arrow{e}
\node{0}
\\
\node{0}
\arrow{e}
\node{\calB(\BC)}
\arrow{e}
\arrow{s}
\node{\calP(\BC)}
\arrow{e,t}{\nu}
\arrow{s}
\node{\BC^* \wedge \BC^*}
\arrow{e}
\arrow{s}
\node{K_2(\BC)}
\arrow{e}
\arrow{s}
\node{0}
\\
\node[2]{0}
\node{0}
\node{0}
\node{0}
\end{diagram}
$$
where

\begin{eqnarray*}
G \wedge G &=& G \otimes_{\BZ} G/(a \otimes b+b \otimes a) \\
\hat\psi(z) &=& \hat\chi(e^{2 \pi i z}) \\
\tau(z) &=& z \wedge 2 \pi i \\
(-\exp \wedge \exp)(a \wedge b) &=& -(e^{2 \pi i a} \wedge e^{2 \pi i b}).
\end{eqnarray*}
In addition, we have the following useful Corollary, from \cite[3.14]{GZ}.

\begin{corollary}
\lbl{cor.Ne}
For $z \in \BC$ we have:
$$
\frac{1}{(2 \pi i)^2} \hat{R}(\hat{\chi}(e^{2 \pi i z}))= z
$$
The restriction
\begin{equation*}
\hat{R}: \Ker(\hatBC \to \calB(\BC)) \longto \BC/\BZ(2)
\end{equation*}
is 1-1.
\end{corollary}

\ifx\undefined\bysame
        \newcommand{\bysame}{\leavevmode\hbox
to3em{\hrulefill}\,}
\fi


\begin{thebibliography}{[EMSS]}

\bibitem[AL]{AL} M. Abouzahra and L. Lewin, 
        {\em The polylogarithm in algebraic number fields},  
        J. Number Theory  {\bf 21}  (1985) 214--244.

\bibitem[B-N]{B-N} D. Bar-Natan,
        {\em Knot Atlas},
        {\tt http://katlas.math.toronto.edu}

\bibitem[BD]{BD} A. Beilinson and P. Deligne,
        {\em Motivic interpretation of the Zagier conjecture connecting 
        polylogarithms and regulators},
        in Motives, Proc. Symp. Pure Math. AMS {\bf 55} (1984) 97--121.

\bibitem[B1]{B1} S. Bloch,
        {\em Higher regulators, algebraic $K$-theory, and zeta functions 
        of elliptic curves}, printed version of the Irvine 1978 lectures. 
        CRM Monograph Series, {\bf 11} AMS, 2000.

\bibitem[B2]{B2} \bysame,
        {\em Algebraic cycles and higher $K$-theory},  
        Adv. in Math.  {\bf 61}  (1986) 267--304.

\bibitem[CG]{CG} O. Costin and S. Garoufalidis,
        {\em Resurgence of 1-dimensional sums of $q$-factorials},
        preprint 2007.

\bibitem[DS]{DS} J.L. Dupont and C.H. Sah,
        {\em Scissors congruences II},
         J. Pure Appl. Algebra  {\bf 25}  (1982) 159--195.

\bibitem[DZ]{DZ} \bysame and C. Zickert, 
        {\em A dilogarithmic formula for the Cheeger-Chern-Simons class},
        Geom. Topol. {\bf 10} (2006) 1347--1372. 

\bibitem[E-V]{E-V} P. Elbaz-Vincent, 
        {\em A short introduction to higher Chow groups},
        in  Transcendental aspects of algebraic cycles,
        London Math. Soc. Lecture Note Ser., {\bf 313} (2004) 171--196. 

\bibitem[FK]{FK} L.D. Faddeev and R.M. Kashaev,
        {\em Quantum dilogarithm},  
        Modern Phys. Lett. A  {\bf 9}  (1994) 427--434.

\bibitem[GL1]{GL1} S. Garoufalidis and T.T.Q. Le,
        {\em The colored Jones function is $q$-holonomic},
        Geom. and Topology {\bf 9} (2005) 1253--1293.

\bibitem[GL2]{GL2} \bysame and T.T.Q. Le,
        {\em Asymptotics of the colored Jones function of a knot},
        preprint 2005 {\tt math.GT/0508100}.

\bibitem[GL3]{GL3} \bysame and \bysame,
        {\em Gevrey series in quantum topology},
        J. Reine Angew. Math., {\bf 618}  (2008) 169--195.

\bibitem[Ga1]{Ga1} \bysame,
        {\em An extended version of additive $K$-theory},
        J. K-Theory  {\bf 4}  (2009) 391--403.

\bibitem[Ga2]{Ga2} \bysame,
        {\em An ansatz for the singularities of hypergeometric multisums},
         Adv. in Appl. Math.  {\bf 41}  (2008) 423--451. 

\bibitem[Ga3]{Ga3} \bysame,
        {\em Chern-Simons theory, analytic continuation and arithmetic},
        Acta Math. Vietnam.  {\bf 33}  (2008) 335--362.

\bibitem[GZ]{GZ} S. Goette and C. Zickert,
        {\em The Extended Bloch Group and the Cheeger-Chern-Simons Class},
         Geom. Topol.  {\bf 11}  (2007) 1623--1635. 

\bibitem[Go1]{Go1} A. Goncharov,
        {\em Polylogarithms and motivic Galois groups}, 
        in Motives, Proc. Symp. Pure Math. AMS {\bf 55} (1984) 43--96.

\bibitem[Go2]{Go2} \bysame,
        {\em Geometry of configurations, polylogarithms and motivic 
        cohomology},
        Adv. in Math. {\bf 114} (1995) 197--318.

\bibitem[Ks]{Ka} R. Kashaev,
        {\em The hyperbolic volume of knots from the quantum dilogarithm},
        Modern Phys. Lett. A {\bf 39} (1997) 269--275.

\bibitem[KLM-S]{KLM-S} M. Kerr, J.D. Lewis and S. M\"uller-Stach,
        {\em The Abel-Jacobi map for higher Chow groups},
        Compos. Math.  {\bf 142}  (2006) 374--396.
 
\bibitem[Ko]{Ko} M. Kontsevich,
        problem proposed in Aarhus, 2006. {\tt http://www.ctqm.au.dk/PL}

\bibitem[KZ]{KZ} \bysame and D. Zagier,
        {\em Periods}, in 
        Mathematics unlimited---2001 and beyond, (2001) 771--808.

\bibitem[Le]{Le} L. Lewin, 
        {\em The inner structure of the dilogarithm in algebraic fields},
        J. Number Theory {\bf 19} (1984) 345--373. 

\bibitem[Na]{Na} W. Nahm,
        {\em Conformal field theory and torsion elements of the Bloch group},
        in Frontiers in Number Theory, Physics and Geometry II, 67--132.

\bibitem[NZ]{NZ} W.D. Neumann and D. Zagier,
        {\em Volumes of hyperbolic three-manifolds},  
        Topology  {\bf 24}  (1985) 307--332.

\bibitem[Ne1]{Ne1} \bysame,
        {\em Extended Bloch group and the Cheeger-Chern-Simons class},  
        Geom. Topol. {\bf 8} (2004) 413--474. 

\bibitem[Ne2]{Ne2} \bysame,
        {\em Hilbert's 3rd problem and invariants of $3$-manifolds},  
        The Epstein birthday schrift,  
        Geom. Topol. Monogr., {\bf 1} (1998) 383--411.

\bibitem[Oe]{Oe} J. Oesterl\'e, 
        {\em Polylogarithmes}, 
        S\'eminaire Bourbaki, Vol. 1992/93.
        Ast\'erisque No. 216 (1993), Exp. No. 762 49--67.


\bibitem[St]{St} R.P. Stanley,
        {\em Enumerative Combinatorics, Volume 1}, 
        Cambridge University Press (1997).

\bibitem[Su1]{Su1} A.A. Suslin, 
        {\em $K_3$ of a field, and the Bloch group},
        Translated in Proc. Steklov Inst. Math. {\bf 4} (1991) 217--239. 

\bibitem[Su2]{Su2} \bysame,
        {\em Algebraic $K$-theory of fields},  
        Proceedings of the International Congress of Mathematicians, 
        Vol. {\bf 1, 2}, Berkeley, (1986)  222--244.

\bibitem[Su3]{Su3} \bysame,
        {\em On the $K$-theory of local fields},
        Journal of Pure and Applied Alg. {\bf 34} (1984) 301--318. 

\bibitem[Th]{Th} W. Thurston, {\em The geometry and topology of 3-manifolds},
        1979 notes, available from MSRI.

\bibitem[Tu]{Tu}
        V.~G. Turaev, {\it Quantum invariants of knots and 3-manifolds}, 
        de Gruyter Studies in Mathematics {\bf 18} Walter de Gruyter, 1994.

\bibitem[WZ]{WZ} H. Wilf and D. Zeilberger,
        {\em An algorithmic proof theory for hypergeometric (ordinary and
        $q$) multisum/integral identities},
        Inventiones Math. {\bf 108} (1992)  575--633.

\bibitem[Za1]{Za1} D. Zagier,
        {\em Polylogarithms, Dedekind zeta functiona and the algebraic 
        $K$-theory of fields}, 
        in Arithmetic Algebraic Geometry, Progr. Math. {\bf 89} (1991) 
        391--430.

\bibitem[Za2]{Za2} \bysame,
        {\em Zagier, Don The dilogarithm function},  
        Frontiers in number theory, physics, and geometry. II,  
        Springer (2007) 3--65.

\bibitem[Z]{Z} D. Zeilberger,
        {\em A holonomic systems approach to special functions identities},
        J. Comput. Appl. Math. {\bf 32} (1990) 321--368.

\bibitem[Zi]{Zi} C. Zickert,
        {\em The extended Bloch group and algebraic K-theory},
        preprint 2009 {\tt arXiv:0910.4005}. 

\end{thebibliography}
\end{document}